\documentclass[10pt]{amsart}
\usepackage{color,colortbl,amsmath,amsfonts,amssymb}

\newtheorem{theorem}{Theorem}[section]
\newtheorem{lemma}[theorem]{Lemma}
\newtheorem{prop}[theorem]{Proposition}
\newtheorem{cor}[theorem]{Corollary}

\theoremstyle{definition}
\newtheorem{definition}[theorem]{Definition}

\numberwithin{equation}{section}

\theoremstyle{remark}

\numberwithin{equation}{section}

\newcommand{\ord}[2]{{\text{ord}_{#1}(#2)}}
\newcommand{\lexp}[2]{e^{{#1}\frac{{#2}\pi i}{13}}}

\newcommand{\curt}{\mathcal F}

\newcommand{\ba}{\begin{eqnarray}}
\newcommand{\ea}{\end{eqnarray}}
\newcommand{\C}{{\mathbb{C}}}
\newcommand{\R}{{\mathfrak{R}}}
\newcommand{\FS}{{\mathfrak{S}}}
\newcommand{\N}{{\mathbb{N}}}
\newcommand{\Z}{{\mathbb{Z}}}
\newcommand{\no}{\noindent}

\def\lcm{\textrm{lcm}}
\newcommand{\jac}[2]{{\left(\frac{#1}{#2}\right)}}

\begin{document}

\title[Distinguished Integers]{On the Collection of Integers that Index the Fixed Points of Maps on the Space of Rational Functions}

\author{Curtis D. Bennett}
\address{Department of Mathematics, Loyola Marymount University, Los Angeles, CA 90045}
\email{cbennett@lmu.edu}

\author{Edward Mosteig}
\address{Department of Mathematics, Loyola Marymount University, Los Angeles, CA 90045}
\email{emosteig@lmu.edu}




\begin{abstract}
Given integers $s$ and $t$, define a function $\phi_{s,t}$ on the
space of all formal complex series expansions by $\phi_{s,t} (\sum
a_n x^n) = \sum a_{sn+t} x^n$.  We define an integer $r$ to be
distinguished with respect to $(s,t)$ if $r$ and $s$ are relatively
prime and and $r \mid t (1 + s + \cdots s^{\ord{r}{s}-1})$.   The
vector space consisting of all rational functions whose Taylor
expansions at zero are fixed by $\phi_{s,t}$ was previously
classified by constructing a basis that is partially indexed by
integers
 that are distinguished with respect
to the pair $(s,t)$.  In this paper, we study the properties of the
set of distinguished integers with respect to $(s,t)$.  In
particular, we demonstrate that the set of distinguished integers
with respect to $(s,t)$ can be written as a union of infinitely many
arithmetic progressions. In addition, we construct another
generating set for the collection of rational functions that are
fixed by $\phi_{s,t}$ and discuss the relationship between this
generating set and the basis that was generated previously.
\end{abstract}

\maketitle







\section{Introduction} \label{intro}
\setcounter{equation}{0}

Consider the space $\FS$ of all formal series with complex
coefficients of the form
$$R(x) = \sum_{n=-\infty}^\infty a_n x^n.$$
Let $\R$ denote the space of rational functions with complex
coefficients. The Taylor expansion at $x =0$ of $R \in \R$ can be
written as a Laurent series, i.e., \ba R(x) & = & \sum_{n \gg
-\infty} a_{n}x^{n} \ea \no where $n \gg -\infty$ denotes the fact
that the coefficients vanish for large negative $n$.

 For $s, t \in \Z$, define the map $\phi_{s,t}: \FS
\to \FS$ by \ba \phi_{s,t}(\sum a_nx^n) = \sum a_{sn+t}x^n. \ea

 When $s$ is positive, consider the restriction  $\phi_{s,t}: \R \to \R$.
The fixed points of $\phi_{s,t}$ are described in \cite{bolimomost}
and \cite{mosteig}, but these points are parameterized by sequences
of integers that are not well understood.  The purpose of this paper
is shed some light on the situation.  Before recalling the results
of \cite{mosteig}, we need a few preliminary definitions.

\begin{definition}
An integer $r \ge 2$ is called {\em distinguished} with respect to
the pair $(s,t)$ if $r$ and $s$ are relatively prime and and $$r
\mid \beta_{s,t}(\ord{r}{s})$$ where $$\beta_{s,t}(k) =  t \left(
\frac{s^k -1}{s-1}\right)$$ and $\ord{r}{s}$ represents the smallest
positive integer such that $s^\ord{r}{s} \equiv 1 \mod r$. We denote
the set of integers distinguished with respect to $(s,t)$ by
$\Omega(s,t)$.
\end{definition}

The description of all the fixed points of $\phi_{s,t}$ requires the
notion of {\em cyclotomic cosets}:  given $n, \, r \in \mathbb{N}$
with  $r \ge 1$ such that $r$ and $s$ are relatively prime, \ba
C_{s,r,n} & = & \{ s^{i}n \, \text{ mod } \, r: \; i \in \mathbb{Z}
\} \ea \no is a finite set called the $s$-cyclotomic coset of $n$
mod $r$.  Define $\Lambda_{s,r}$ to be a complete collection of
coset representatives (all chosen to be less than $r$); i.e., for
all $n \in \N$, there exists a unique $n' \in \Lambda_{s,r}$ such
that $C_{s,r,n} = C_{s,r,n'}$.  For $r,n \ge 1$, define \ba
\psi_{s,t,r,n}(x)  =  \sum_{j=1}^{\ord{r}{s}} \phi_{s,t}^{(j)}\left(
\frac{1}{1-\omega_r^nx}\right) =  \sum_{j=1}^{\ord{r}{s}}
\frac{\omega_r^{n{\beta_{s,t}(j)}}}{1 - \omega_r^{ns^j} x}
 , \ea
where  $\omega_s = e^{2\pi i/s}$ and $\phi_{s,t}^{(j)}$ represents
the $j$-th iterate of the function $\phi_{s,t}$.  When $n=0$, the
function  $\psi_{s,t,r,n}$ is defined to be $1/(1-x)$.  We recall
the following result from \cite{mosteig}.

\begin{theorem}\label{edbasis}
Suppose $s\ge 2$ and $1 \le t \le s-2$.  The function $1/(1-x)$
together with the collection of all $\psi_{s,t,r,n}$ where $r$ is
distinguished with respect to $(s,t)$ and $n \in \Lambda_{s,r}$ is
relatively prime to $r$ form a basis for the set of all rational
functions that are fixed under the transformation $\phi_{s,t}$.
\end{theorem}

Although this theorem provides us with a basis for the space of
rational functions fixed by $\phi_{s,t}$, it is somewhat
unsatisfactory in that it does not give us a good sense of what it
means for an integer $r$ to be distinguished with respect to the
pair $(s,t)$.  It was shown in \cite{mosteig} that the collection of
integers that are distinguished with respect to $(s,t)$ has infinite
cardinality, but that is pretty much the limit of what was
discussed.  In this paper, we explore one of the questions posed in
\cite{mosteig}, namely, whether or not $\Omega(s,t)$ is a union of
arithmetic sequences.  We begin by examining this problem in Section
\ref{31} and show, among other results that, indeed, $\Omega(3,1)$
is an infinite union of arithmetic sequences.  In Section \ref{st},
we generalize the results of Section \ref{31} to the case when
working with an arbitrary pair $(s,t)$.  In particular, we show that
$\Omega(s,t)$ is a union of arithmetic sequences, and then we
provide conditions for when multiples of a particular form of a
fixed integer are distinguished with respect to $(s,t)$.

In the course of studying distinguished integers with respect to a
given pair $(s,t)$, another collection of rational functions that
span that space of functions fixed by $\phi_{s,t}$ was discovered.
In Section \ref{spanset}, we describe this spanning set and discuss
its relationship to the collection of functions of the form
$\psi_{s,t,r,n}$.

\section{Distinguished with Respect to $(3,1)$} \label{31} \setcounter{equation}{0}

In this section, we will examine the special case of
$(3,1)$-distinguished integers.  We begin with this case, as it is
the simplest interesting case.  Moreover, the experimental data in this case
suggests a number of avenues for investigation.  From the analysis
of the $(3,1)$ case, we can discover several interesting
propositions, some of which we generalize in the next section.

The table below shows all the integers up to 204 that are
distinguished (shaded) with respect to $(3,1)$.

\begin{center}
\begin{tabular}{|c|c|c|c|c|c|c|c|c|c|c|c|}
  \hline
 \multicolumn{1}{|>{\columncolor[rgb]{0.7,0.7,0.7}}c|}{1} & \multicolumn{1}{>{\columncolor[rgb]{1,1,1}}c|}{2}
 &  \multicolumn{1}{>{\columncolor[rgb]{1,1,1}}c|}{3} &  \multicolumn{1}{>{\columncolor[rgb]{0.7,0.7,0.7}}c|}{4}
 &  \multicolumn{1}{>{\columncolor[rgb]{0.7,0.7,0.7}}c|}{5} &  \multicolumn{1}{>{\columncolor[rgb]{1,1,1}}c|}{6}
 &  \multicolumn{1}{>{\columncolor[rgb]{0.7,0.7,0.7}}c|}{7} &  \multicolumn{1}{>{\columncolor[rgb]{1,1,1}}c|}{8}
 &  \multicolumn{1}{>{\columncolor[rgb]{1,1,1}}c|}{9} &  \multicolumn{1}{>{\columncolor[rgb]{0.7,0.7,0.7}}c|}{10}
 &  \multicolumn{1}{>{\columncolor[rgb]{0.7,0.7,0.7}}c|}{11} &  \multicolumn{1}{>{\columncolor[rgb]{1,1,1}}c|}{12}
 \\
  \hline
 \multicolumn{1}{|>{\columncolor[rgb]{0.7,0.7,0.7}}c|}{13} & \multicolumn{1}{>{\columncolor[rgb]{0.7,0.7,0.7}}c|}{14}
&  \multicolumn{1}{>{\columncolor[rgb]{1,1,1}}c|}{15} &
\multicolumn{1}{>{\columncolor[rgb]{1,1,1}}c|}{16} &
\multicolumn{1}{>{\columncolor[rgb]{0.7,0.7,0.7}}c|}{17} &
\multicolumn{1}{>{\columncolor[rgb]{1,1,1}}c|}{18} &
\multicolumn{1}{>{\columncolor[rgb]{0.7,0.7,0.7}}c|}{19} &
\multicolumn{1}{>{\columncolor[rgb]{0.7,0.7,0.7}}c|}{20} &
\multicolumn{1}{>{\columncolor[rgb]{1,1,1}}c|}{21} &
\multicolumn{1}{>{\columncolor[rgb]{1,1,1}}c|}{22} &
\multicolumn{1}{>{\columncolor[rgb]{0.7,0.7,0.7}}c|}{23} &
\multicolumn{1}{>{\columncolor[rgb]{1,1,1}}c|}{24}  \\
  \hline
 \multicolumn{1}{|>{\columncolor[rgb]{0.7,0.7,0.7}}c|}{25} &
 \multicolumn{1}{>{\columncolor[rgb]{1,1,1}}c|}{26}
&  \multicolumn{1}{>{\columncolor[rgb]{1,1,1}}c|}{27} &
\multicolumn{1}{>{\columncolor[rgb]{0.7,0.7,0.7}}c|}{28} &
\multicolumn{1}{>{\columncolor[rgb]{0.7,0.7,0.7}}c|}{29} &
\multicolumn{1}{>{\columncolor[rgb]{1,1,1}}c|}{30} &
\multicolumn{1}{>{\columncolor[rgb]{0.7,0.7,0.7}}c|}{31} &
\multicolumn{1}{>{\columncolor[rgb]{1,1,1}}c|}{32} &
\multicolumn{1}{>{\columncolor[rgb]{1,1,1}}c|}{33} &
\multicolumn{1}{>{\columncolor[rgb]{0.7,0.7,0.7}}c|}{34} &
\multicolumn{1}{>{\columncolor[rgb]{0.7,0.7,0.7}}c|}{35} &
\multicolumn{1}{>{\columncolor[rgb]{1,1,1}}c|}{36}  \\
  \hline
 \multicolumn{1}{|>{\columncolor[rgb]{0.7,0.7,0.7}}c|}{37} &
 \multicolumn{1}{>{\columncolor[rgb]{0.7,0.7,0.7}}c|}{38}
&  \multicolumn{1}{>{\columncolor[rgb]{1,1,1}}c|}{39} &
\multicolumn{1}{>{\columncolor[rgb]{0.7,0.7,0.7}}c|}{40} &
\multicolumn{1}{>{\columncolor[rgb]{0.7,0.7,0.7}}c|}{41} &
\multicolumn{1}{>{\columncolor[rgb]{1,1,1}}c|}{42} &
\multicolumn{1}{>{\columncolor[rgb]{0.7,0.7,0.7}}c|}{43} &
\multicolumn{1}{>{\columncolor[rgb]{0.7,0.7,0.7}}c|}{44} &
\multicolumn{1}{>{\columncolor[rgb]{1,1,1}}c|}{45} &
\multicolumn{1}{>{\columncolor[rgb]{1,1,1}}c|}{46} &
\multicolumn{1}{>{\columncolor[rgb]{0.7,0.7,0.7}}c|}{47} &
\multicolumn{1}{>{\columncolor[rgb]{1,1,1}}c|}{48}  \\
  \hline
   \multicolumn{1}{|>{\columncolor[rgb]{0.7,0.7,0.7}}c|}{49} &
 \multicolumn{1}{>{\columncolor[rgb]{0.7,0.7,0.7}}c|}{50}
&  \multicolumn{1}{>{\columncolor[rgb]{1,1,1}}c|}{51} &
\multicolumn{1}{>{\columncolor[rgb]{0.7,0.7,0.7}}c|}{52} &
\multicolumn{1}{>{\columncolor[rgb]{0.7,0.7,0.7}}c|}{53} &
\multicolumn{1}{>{\columncolor[rgb]{1,1,1}}c|}{54} &
\multicolumn{1}{>{\columncolor[rgb]{0.7,0.7,0.7}}c|}{55} &
\multicolumn{1}{>{\columncolor[rgb]{1,1,1}}c|}{56} &
\multicolumn{1}{>{\columncolor[rgb]{1,1,1}}c|}{57} &
\multicolumn{1}{>{\columncolor[rgb]{0.7,0.7,0.7}}c|}{58} &
\multicolumn{1}{>{\columncolor[rgb]{0.7,0.7,0.7}}c|}{59} &
\multicolumn{1}{>{\columncolor[rgb]{1,1,1}}c|}{60}  \\
  \hline
 \multicolumn{1}{|>{\columncolor[rgb]{0.7,0.7,0.7}}c|}{61} &
 \multicolumn{1}{>{\columncolor[rgb]{0.7,0.7,0.7}}c|}{62}
&  \multicolumn{1}{>{\columncolor[rgb]{1,1,1}}c|}{63} &
\multicolumn{1}{>{\columncolor[rgb]{1,1,1}}c|}{64} &
\multicolumn{1}{>{\columncolor[rgb]{0.7,0.7,0.7}}c|}{65} &
\multicolumn{1}{>{\columncolor[rgb]{1,1,1}}c|}{66} &
\multicolumn{1}{>{\columncolor[rgb]{0.7,0.7,0.7}}c|}{67} &
\multicolumn{1}{>{\columncolor[rgb]{0.7,0.7,0.7}}c|}{68} &
\multicolumn{1}{>{\columncolor[rgb]{1,1,1}}c|}{69} &
\multicolumn{1}{>{\columncolor[rgb]{0.7,0.7,0.7}}c|}{70} &
\multicolumn{1}{>{\columncolor[rgb]{0.7,0.7,0.7}}c|}{71} &
\multicolumn{1}{>{\columncolor[rgb]{1,1,1}}c|}{72}  \\
  \hline
   \multicolumn{1}{|>{\columncolor[rgb]{0.7,0.7,0.7}}c|}{73} &
 \multicolumn{1}{>{\columncolor[rgb]{0.7,0.7,0.7}}c|}{74}
&  \multicolumn{1}{>{\columncolor[rgb]{1,1,1}}c|}{75} &
\multicolumn{1}{>{\columncolor[rgb]{0.7,0.7,0.7}}c|}{76} &
\multicolumn{1}{>{\columncolor[rgb]{0.7,0.7,0.7}}c|}{77} &
\multicolumn{1}{>{\columncolor[rgb]{1,1,1}}c|}{78} &
\multicolumn{1}{>{\columncolor[rgb]{0.7,0.7,0.7}}c|}{79} &
\multicolumn{1}{>{\columncolor[rgb]{1,1,1}}c|}{80} &
\multicolumn{1}{>{\columncolor[rgb]{1,1,1}}c|}{81} &
\multicolumn{1}{>{\columncolor[rgb]{0.7,0.7,0.7}}c|}{82} &
\multicolumn{1}{>{\columncolor[rgb]{0.7,0.7,0.7}}c|}{83} &
\multicolumn{1}{>{\columncolor[rgb]{1,1,1}}c|}{84}  \\
  \hline
   \multicolumn{1}{|>{\columncolor[rgb]{0.7,0.7,0.7}}c|}{85} &
 \multicolumn{1}{>{\columncolor[rgb]{0.7,0.7,0.7}}c|}{86}
&  \multicolumn{1}{>{\columncolor[rgb]{1,1,1}}c|}{87} &
\multicolumn{1}{>{\columncolor[rgb]{1,1,1}}c|}{88} &
\multicolumn{1}{>{\columncolor[rgb]{0.7,0.7,0.7}}c|}{89} &
\multicolumn{1}{>{\columncolor[rgb]{1,1,1}}c|}{90} &
\multicolumn{1}{>{\columncolor[rgb]{0.7,0.7,0.7}}c|}{91} &
\multicolumn{1}{>{\columncolor[rgb]{0.7,0.7,0.7}}c|}{92} &
\multicolumn{1}{>{\columncolor[rgb]{1,1,1}}c|}{93} &
\multicolumn{1}{>{\columncolor[rgb]{1,1,1}}c|}{94} &
\multicolumn{1}{>{\columncolor[rgb]{0.7,0.7,0.7}}c|}{95} &
\multicolumn{1}{>{\columncolor[rgb]{1,1,1}}c|}{96}  \\
  \hline
 \multicolumn{1}{|>{\columncolor[rgb]{0.7,0.7,0.7}}c|}{97} &
 \multicolumn{1}{>{\columncolor[rgb]{0.7,0.7,0.7}}c|}{98}
&  \multicolumn{1}{>{\columncolor[rgb]{1,1,1}}c|}{99} &
\multicolumn{1}{>{\columncolor[rgb]{0.7,0.7,0.7}}c|}{100} &
\multicolumn{1}{>{\columncolor[rgb]{0.7,0.7,0.7}}c|}{101} &
\multicolumn{1}{>{\columncolor[rgb]{1,1,1}}c|}{102} &
\multicolumn{1}{>{\columncolor[rgb]{0.7,0.7,0.7}}c|}{103} &
\multicolumn{1}{>{\columncolor[rgb]{1,1,1}}c|}{104} &
\multicolumn{1}{>{\columncolor[rgb]{1,1,1}}c|}{105} &
\multicolumn{1}{>{\columncolor[rgb]{0.7,0.7,0.7}}c|}{106} &
\multicolumn{1}{>{\columncolor[rgb]{0.7,0.7,0.7}}c|}{107} &
\multicolumn{1}{>{\columncolor[rgb]{1,1,1}}c|}{108}  \\
  \hline
   \multicolumn{1}{|>{\columncolor[rgb]{0.7,0.7,0.7}}c|}{109} &
 \multicolumn{1}{>{\columncolor[rgb]{0.7,0.7,0.7}}c|}{110}
&  \multicolumn{1}{>{\columncolor[rgb]{1,1,1}}c|}{111} &
\multicolumn{1}{>{\columncolor[rgb]{1,1,1}}c|}{112} &
\multicolumn{1}{>{\columncolor[rgb]{0.7,0.7,0.7}}c|}{113} &
\multicolumn{1}{>{\columncolor[rgb]{1,1,1}}c|}{114} &
\multicolumn{1}{>{\columncolor[rgb]{0.7,0.7,0.7}}c|}{115} &
\multicolumn{1}{>{\columncolor[rgb]{0.7,0.7,0.7}}c|}{116} &
\multicolumn{1}{>{\columncolor[rgb]{1,1,1}}c|}{117} &
\multicolumn{1}{>{\columncolor[rgb]{1,1,1}}c|}{118} &
\multicolumn{1}{>{\columncolor[rgb]{0.7,0.7,0.7}}c|}{119} &
\multicolumn{1}{>{\columncolor[rgb]{1,1,1}}c|}{120}  \\
  \hline
 \multicolumn{1}{|>{\columncolor[rgb]{0.7,0.7,0.7}}c|}{121} &
 \multicolumn{1}{>{\columncolor[rgb]{0.7,0.7,0.7}}c|}{122}
&  \multicolumn{1}{>{\columncolor[rgb]{1,1,1}}c|}{123} &
\multicolumn{1}{>{\columncolor[rgb]{0.7,0.7,0.7}}c|}{124} &
\multicolumn{1}{>{\columncolor[rgb]{0.7,0.7,0.7}}c|}{125} &
\multicolumn{1}{>{\columncolor[rgb]{1,1,1}}c|}{126} &
\multicolumn{1}{>{\columncolor[rgb]{0.7,0.7,0.7}}c|}{127} &
\multicolumn{1}{>{\columncolor[rgb]{1,1,1}}c|}{128} &
\multicolumn{1}{>{\columncolor[rgb]{1,1,1}}c|}{129} &
\multicolumn{1}{>{\columncolor[rgb]{0.7,0.7,0.7}}c|}{130} &
\multicolumn{1}{>{\columncolor[rgb]{0.7,0.7,0.7}}c|}{131} &
\multicolumn{1}{>{\columncolor[rgb]{1,1,1}}c|}{132}  \\
  \hline
   \multicolumn{1}{|>{\columncolor[rgb]{0.7,0.7,0.7}}c|}{133} &
 \multicolumn{1}{>{\columncolor[rgb]{0.7,0.7,0.7}}c|}{134}
&  \multicolumn{1}{>{\columncolor[rgb]{1,1,1}}c|}{135} &
\multicolumn{1}{>{\columncolor[rgb]{0.7,0.7,0.7}}c|}{136} &
\multicolumn{1}{>{\columncolor[rgb]{0.7,0.7,0.7}}c|}{137} &
\multicolumn{1}{>{\columncolor[rgb]{1,1,1}}c|}{138} &
\multicolumn{1}{>{\columncolor[rgb]{0.7,0.7,0.7}}c|}{139} &
\multicolumn{1}{>{\columncolor[rgb]{0.7,0.7,0.7}}c|}{140} &
\multicolumn{1}{>{\columncolor[rgb]{1,1,1}}c|}{141} &
\multicolumn{1}{>{\columncolor[rgb]{1,1,1}}c|}{142} &
\multicolumn{1}{>{\columncolor[rgb]{0.7,0.7,0.7}}c|}{143} &
\multicolumn{1}{>{\columncolor[rgb]{1,1,1}}c|}{144}  \\
  \hline
   \multicolumn{1}{|>{\columncolor[rgb]{0.7,0.7,0.7}}c|}{145} &
 \multicolumn{1}{>{\columncolor[rgb]{0.7,0.7,0.7}}c|}{146}
&  \multicolumn{1}{>{\columncolor[rgb]{1,1,1}}c|}{147} &
\multicolumn{1}{>{\columncolor[rgb]{0.7,0.7,0.7}}c|}{148} &
\multicolumn{1}{>{\columncolor[rgb]{0.7,0.7,0.7}}c|}{149} &
\multicolumn{1}{>{\columncolor[rgb]{1,1,1}}c|}{150} &
\multicolumn{1}{>{\columncolor[rgb]{0.7,0.7,0.7}}c|}{151} &
\multicolumn{1}{>{\columncolor[rgb]{1,1,1}}c|}{152} &
\multicolumn{1}{>{\columncolor[rgb]{1,1,1}}c|}{153} &
\multicolumn{1}{>{\columncolor[rgb]{0.7,0.7,0.7}}c|}{154} &
\multicolumn{1}{>{\columncolor[rgb]{0.7,0.7,0.7}}c|}{155} &
\multicolumn{1}{>{\columncolor[rgb]{1,1,1}}c|}{156}  \\
  \hline
 \multicolumn{1}{|>{\columncolor[rgb]{0.7,0.7,0.7}}c|}{157} &
 \multicolumn{1}{>{\columncolor[rgb]{0.7,0.7,0.7}}c|}{158}
&  \multicolumn{1}{>{\columncolor[rgb]{1,1,1}}c|}{159} &
\multicolumn{1}{>{\columncolor[rgb]{1,1,1}}c|}{160} &
\multicolumn{1}{>{\columncolor[rgb]{0.7,0.7,0.7}}c|}{161} &
\multicolumn{1}{>{\columncolor[rgb]{1,1,1}}c|}{162} &
\multicolumn{1}{>{\columncolor[rgb]{0.7,0.7,0.7}}c|}{163} &
\multicolumn{1}{>{\columncolor[rgb]{0.7,0.7,0.7}}c|}{164} &
\multicolumn{1}{>{\columncolor[rgb]{1,1,1}}c|}{165} &
\multicolumn{1}{>{\columncolor[rgb]{1,1,1}}c|}{166} &
\multicolumn{1}{>{\columncolor[rgb]{0.7,0.7,0.7}}c|}{167} &
\multicolumn{1}{>{\columncolor[rgb]{1,1,1}}c|}{168}  \\
  \hline
   \multicolumn{1}{|>{\columncolor[rgb]{0.7,0.7,0.7}}c|}{169} &
 \multicolumn{1}{>{\columncolor[rgb]{0.7,0.7,0.7}}c|}{170}
&  \multicolumn{1}{>{\columncolor[rgb]{1,1,1}}c|}{171} &
\multicolumn{1}{>{\columncolor[rgb]{0.7,0.7,0.7}}c|}{172} &
\multicolumn{1}{>{\columncolor[rgb]{0.7,0.7,0.7}}c|}{173} &
\multicolumn{1}{>{\columncolor[rgb]{1,1,1}}c|}{174} &
\multicolumn{1}{>{\columncolor[rgb]{0.7,0.7,0.7}}c|}{175} &
\multicolumn{1}{>{\columncolor[rgb]{1,1,1}}c|}{176} &
\multicolumn{1}{>{\columncolor[rgb]{1,1,1}}c|}{177} &
\multicolumn{1}{>{\columncolor[rgb]{0.7,0.7,0.7}}c|}{178} &
\multicolumn{1}{>{\columncolor[rgb]{0.7,0.7,0.7}}c|}{179} &
\multicolumn{1}{>{\columncolor[rgb]{1,1,1}}c|}{180}  \\
  \hline
 \multicolumn{1}{|>{\columncolor[rgb]{0.7,0.7,0.7}}c|}{181} &
 \multicolumn{1}{>{\columncolor[rgb]{0.7,0.7,0.7}}c|}{182}
&  \multicolumn{1}{>{\columncolor[rgb]{1,1,1}}c|}{183} &
\multicolumn{1}{>{\columncolor[rgb]{1,1,1}}c|}{184} &
\multicolumn{1}{>{\columncolor[rgb]{0.7,0.7,0.7}}c|}{185} &
\multicolumn{1}{>{\columncolor[rgb]{1,1,1}}c|}{186} &
\multicolumn{1}{>{\columncolor[rgb]{0.7,0.7,0.7}}c|}{187} &
\multicolumn{1}{>{\columncolor[rgb]{0.7,0.7,0.7}}c|}{188} &
\multicolumn{1}{>{\columncolor[rgb]{1,1,1}}c|}{189} &
\multicolumn{1}{>{\columncolor[rgb]{0.7,0.7,0.7}}c|}{190} &
\multicolumn{1}{>{\columncolor[rgb]{0.7,0.7,0.7}}c|}{191} &
\multicolumn{1}{>{\columncolor[rgb]{1,1,1}}c|}{192}  \\
  \hline
   \multicolumn{1}{|>{\columncolor[rgb]{0.7,0.7,0.7}}c|}{193} &
 \multicolumn{1}{>{\columncolor[rgb]{0.7,0.7,0.7}}c|}{194}
&  \multicolumn{1}{>{\columncolor[rgb]{1,1,1}}c|}{195} &
\multicolumn{1}{>{\columncolor[rgb]{0.7,0.7,0.7}}c|}{196} &
\multicolumn{1}{>{\columncolor[rgb]{0.7,0.7,0.7}}c|}{197} &
\multicolumn{1}{>{\columncolor[rgb]{1,1,1}}c|}{198} &
\multicolumn{1}{>{\columncolor[rgb]{0.7,0.7,0.7}}c|}{199} &
\multicolumn{1}{>{\columncolor[rgb]{0.7,0.7,0.7}}c|}{200} &
\multicolumn{1}{>{\columncolor[rgb]{1,1,1}}c|}{201} &
\multicolumn{1}{>{\columncolor[rgb]{0.7,0.7,0.7}}c|}{202} &
\multicolumn{1}{>{\columncolor[rgb]{0.7,0.7,0.7}}c|}{203} &
\multicolumn{1}{>{\columncolor[rgb]{1,1,1}}c|}{204}  \\
  \hline \end{tabular}
\end{center}

Upon examining this table, it seems rather likely that all positive
integers that are congruent to either 1 or 5 modulo 6 must be
distinguished with respect to $(3,1)$.  Moreover, it appears that
all positive integers congruent to 4, 10, 14, or 20 modulo 24 must
be $(3,1)$-distinguished.  In fact, both statements are true and are
mentioned in \cite{mosteig}, and below we will provide proofs. Our
general methods only yield the case of $10$ modulo $60$  and $14$
modulo $88$ rather than modulo $24$.  To obtain the proofs for 10
and 14 modulo 24, we use quadratic reciprocity.  One would hope for
simpler proofs of these last two cases, and the interested reader is
encouraged to look for such proofs.

Of the remaining distinguished integers in the table above, 40 is
the smallest. Again, the pattern seems promising. Multiply the
previous modulus by four to obtain 96. Jumping to conclusions, it
seems likely that all positive integers congruent to 40 modulo 96
must be distinguished. In fact, 40, 136, 232, 328, 424, and 520 are
all distinguished with respect to $(3,1)$, but 616 is not!   This
surprising gap leads to some interesting questions. In light of this
example, it is not clear whether the collection of all integers that
are distinguished with respect to $(3,1)$ can be written as a
(possibly infinite) union of congruence classes, and we now turn to answer this question.

We begin by establishing that odd positive integers relatively prime
to $6$ are $(3,1)$-distinguished.

\begin{lemma}
Every integer congruent to 1 or 5 modulo 6 must be distinguished
with respect to $(3,1)$.
\end{lemma}

\begin{proof}
Suppose $r$  is congruent to 1 or 5 modulo 6; that is, $r$ is
relatively prime to $6$. By definition, $3^\ord{r}{3} \equiv 1 \mod
r$, and so $r \mid 3^\ord{r}{3} - 1$. Now, $3^\ord{r}{3}-1$ is even
and $r$ is odd, and so $r \mid \frac{3^\ord{r}{3} - 1}{3-1}$.
\end{proof}

It is a bit trickier to justify that positive integers in the
equivalence classes modulo 24 containing 4 and 20 are distinguished
with respect to $(3,1)$.  Since these equivalence classes consist
solely of even integers, we must employ a different argument.  To
begin, we note the following result.

\begin{lemma} \label{ord31reduction}
Let $r$ be a positive integer that is relatively prime to $3$. Then
$r$ is distinguished with respect to $(3,1)$ if and only if
$\ord{r}{3} = \ord{2r}{3}$.
\end{lemma}

\begin{proof}
If  $r$ is distinguished with respect to $(3,1)$, then $r \mid
(3^\ord{r}{3}-1)/(3-1)$, and so $3^\ord{r}{3} \equiv 1 \mod 2r$.
 Thus $\ord{r}{3} \ge \ord{2r}{3}$, and since the
reverse inequality always holds,  $\ord{r}{3} = \ord{2r}{3}$.

Conversely, suppose $\ord{r}{3} = \ord{2r}{3}$.  Since
$3^\ord{2r}{3} \equiv 1 \mod 2r,$ it follows that $3^\ord{r}{3}
\equiv 1 \mod 2r,$ and so $2r \mid 3^\ord{r}{3} -1$, in which case
$r \mid  (3^\ord{r}{3} -1)/(3-1)$.
\end{proof}

We note that if $m$ and $n$ are relatively prime, then
$\ord{mn}{a}=\lcm(\ord{m}{a},\ord{n}{a})$ (which follows as the
group of multiplicative units modulo $mn$ is a direct product of the
group of units modulo $m$ and the group of units modulo $n$).  As a
result, by writing $r=2^t k$ with $\gcd(2,k)=1$, we see that the
validity of the equation $\ord{r}{3}=\ord{2r}{3}$ hinges on the
relationship between $\ord{2^t}{3}$, $\ord{2^{t+1}}{3}$ and
$\ord{k}{3}$.  We begin our more general analysis by examining the
relationship between the first two of these quantities.

\begin{prop}\label{ord2ell}
For $\ell \ge 3$, $\ord{2^\ell}{3}  = 2^{\ell-2}$.
\end{prop}

\begin{proof}
We will prove that $3^{2^\ell} - 1 \equiv 2^{\ell+2} \mod
2^{\ell+3}$ by induction, from which the result follows. It is
easily verified that this holds for $\ell = 3$, and so we assume
$3^{2^\ell} - 1 \equiv 2^{\ell+2} \mod 2^{\ell+3}$ for a particular
value of $\ell$. From this, it follows that for some $q\in \N$,
$3^{2^\ell} - 1 - 2^{\ell+2} = 2^{\ell+3} q$. Moreover, for all
$\ell \ge 1$, $3^{2^\ell} \equiv 1 \mod 4$, and so $3^{2^\ell}+1
\equiv 2 \mod 4$; thus $3^{2^\ell}+1 = 4q'+2$ for some $q' \in \N$.
Thus, $3^{2^{\ell+1}}- 1 = (3^{2^\ell}+1)(3^{2^\ell}-1) =
(4q'+2)(2^{\ell+2} + 2^{\ell+3}q) = 2^{\ell+3}(1+q)(1+2q)$, and so
$3^{2^{\ell+1}}-1 \equiv 2^{\ell+3} \mod 2^{\ell+4}$.
\end{proof}

We note that $\ord{2}{3}=1$ and $\ord{4}{3}=2=\ord{8}{3}$; from the
latter we have that $4$ is necessarily $(3,1)$-distinguished. The
following lemma generalizes the case of $4$ to numbers of the form
$2^\ell k$ with $\gcd(6,k)=1$.

\begin{lemma}\label{dist3}
Given a positive integer $r = 2^\ell k$ such that $\ell \ge 3$ and
$\gcd(k,6) = 1$, $r$ is distinguished with respect to $(3,1)$ if and
only if $2^{\ell-1} \mid \ord{k}{3}$.  Moreover, if $r=2k$ with
$\gcd(k,6)=1$, then $r$ is $(3,1)$-distinguished if and only if $2
\mid \ord{k}{3}$, and if $r=4k$ with $\gcd(k,6)=1$, then $r$
$(3,1)$-distinguished.
\end{lemma}

\begin{proof}
By Lemma \ref{ord31reduction}, $r$ is distinguished with respect to
$(3,1)$ if and only if $\ord{r}{3} = \ord{2r}{3}$. If $r$ is of the
form $r = 2^\ell k$ such that $\ell \ge 3$ and $\gcd(k,6) = 1$, then
$\ord{r}{3} = \lcm(\ord{2^\ell}{s},\ord{k}{3})$. Moreover,
$\ord{2r}{3} = \lcm(\ord{2^{\ell+1}}{s},\ord{k}{3})$ and so $r$ is
distinguished with respect to $(3,1)$ if and only if
$\lcm(\ord{2^\ell}{3},\ord{k}{3}) =
\lcm(\ord{2^{\ell+1}}{3},\ord{k}{3})$. Since $\ord{2^{\ell+1}}{3}=2\cdot\ord{2^\ell}{3}$, this condition holds whenever
$\ord{2^{\ell+1}}{3} \mid \ord{k}{3}$.  By Proposition
\ref{ord2ell}, this is equivalent to $2^{\ell-1} \mid \ord{k}{3}$.

For the other two cases, we note that if $r=2k$ with $\gcd(6,k)=1$,
then $\ord{r}{3}=\ord{k}{3}$, while $\ord{2r}{3}=\lcm(2,\ord{k}{3})$
so that we have equality if and only if $\ord{k}{3}$ is even.
Alternatively, if $r=4k$ with $\gcd(6,k)=1$, then
$\ord{r}{3}=\lcm(2,\ord{k}{3})$, while
$\ord{2r}{3}=\lcm(2,\ord{k}{3})$ as $\ord{8}{3}=2=\ord{4}{3}$.
\end{proof}

Since any number $r$ congruent to $4$ or $20$ modulo $24$ is of the
form $r=4k$ where $\gcd(6,k)=1$, the above lemma implies that all
such numbers are $(3,1)$-distinguished.  Sadly, the appearances of
arithmetic series $10$ and $14$ modulo $24$ in our chart are still
hard to explain.

The following results answer the original question concerning
whether all the distinguished integers with respect to $(3,1)$ can
be written as an infinite union of arithmetic sequences.

\begin{cor} \label{31congr6r}
Suppose $r$ is $(3,1)$-distinguished.  Then all integers congruent
to $r$ or $5r$ modulo $6r$ are also $(3,1)$-distinguished.
\end{cor}
\begin{proof}
Suppose $r$ is $(3,1)$ distinguished, and write $r=2^t k$ where
$\gcd(6,k)=1$.  It follows from Lemma~\ref{ord31reduction} that
$\ord{r}{3}=\ord{2r}{3}$.  This implies that
$$
\lcm(\ord{2^t}{3},\ord{k}{3})=\lcm(\ord{2^{t+1}}{3},\ord{k}{3}).
$$
Suppose $r'=r+6rm$, for some integer $m$.  Then
\begin{eqnarray*}
\ord{r'}{3} & = & \ord{r+6rm}{3} \\
 & = & \ord{2^tk(1+6m)}{3} \\
 & = & \lcm(\ord{2^t}{3},\ord{k(1+6m)}{3}).
 \end{eqnarray*}
Similarly, $\ord{2r'}{3}=\lcm(\ord{2^{t+1}}{3},\ord{k(1+6m)}{3})$.
Note that $\ord{k}{3}$ divides $\ord{k(1+6m)}{3}$.  However, if
$\lcm(x,z)=\lcm(y,z)$, then it must be the case that
$\lcm(x,za)=\lcm(y,za)$ when $x,y,z,a\in\mathbb{Z}$.  Consequently,
letting $x=\ord{2^t}{3}$, $y=\ord{2^{t+1}}{3}$, $z=\ord{k}{3}$ and
$za=\ord{k(1+6m)}{3}$, we have that $\ord{r'}{3}=\ord{2r'}{3}$ so
that $r'$ is $(3,1)$-distinguished. For the case of $5r$ modulo
$6r$, we note that the only change in the above is that $1+6m$ is
replaced by $5+6m$. Thus if $r'=5r+6mr$ for some $m\in\mathbb{N}$,
then $r'$ is $(3,1)$ distinguished too.
\end{proof}

We could have used Lemma~\ref{dist3} and a case-by-case analysis for
this, but the above argument is both more elegant and more easily
generalized. Note that since $10$ is $(3,1)$-distinguished,
Corollary~\ref{31congr6r} implies that if $r\equiv 10 \mod 60$ then
$r$ is also $(3,1)$-distinguished.  Similarly, we know that integers
congruent to $14$ modulo $84$ are $(3,1)$-distinguished.  However,
neither of these quickly leads to an argument for $10$ or $14$
modulo $24$.  On the other hand, we do obtain the following result:

\begin{cor}\label{31progression}
The set of all $(3,1)$-distinguished integers can be written as an
infinite union of arithmetic progressions.
\end{cor}
\begin{proof}
By Corollary~\ref{31congr6r} every $(3,1)$-distinguished integer $r$
lies in the arithmetic progression $(r+6rm)_{m=1}^\infty$.
\end{proof}

This corollary answers our initial question, but as our difficulty
with $10$ and $14$ show, the answer is not entirely satisfactory.
For completeness, we will justify that all positive integers
congruent to 10 modulo 24 are $(3,1)$-distinguished.  The proof for
14 is similar.

\begin{prop}\label{10mod24}
Suppose $r \equiv 10 \mod 24$ where $r>0$.  Then $r$ is
$(3,1)$-distinguished.
\end{prop}

\begin{proof}
By Lemma~\ref{ord31reduction} we need to show that
$\ord{r}{3}=\ord{2r}{3}$.  Writing $r = 10+24k$ where $k$ is a
nonnegative integer, this corresponds to showing
$$
\ord{10+24k}{3}=\lcm(\ord{2}{3},\ord{5+12k}{3})=\ord{5+12k}{3}
$$
is equal to
$$
\ord{20+48k}{3}=\lcm(\ord{4}{3},\ord{5+12k}{3})=\lcm(2,\ord{5+12k}{3}).
$$
This follows if and only if $\ord{5+12k}{3}$ is even.  Consequently,
we simply need to show that $3^{2n+1}\not\equiv 1\mod{5+12k}$ for
any $n$.  As $3(2+4k)\equiv 1 \mod{5+12k}$, this corresponds to
showing that $3^{2n}=(3^n)^2\not\equiv 2+4k \mod{5+12k}$ for any
$n$.  Thus the result follows if we show that $2+4k$ is not a square
modulo $5+12k$. We turn to quadratic reciprocity for this result.
 Recall that if $a$ is a square mod $b$, then the Jacobi symbol $\left( \frac{a}{b}
 \right) = 1$.  Using the algebra of Jacobi symbols (see \cite{grossw}, for
example),
\begin{eqnarray*}
\jac{2+4k}{5+12k} & = & \jac{2}{5+12k}\jac{1+2k}{5+12k}\\
  & = & (-1)^{((5+12k)^2-1)/8}\jac{1+2k}{5+12k} \\
  & = & (-1)^{k+1}\jac{1+2k}{5+12k},
\end{eqnarray*}
and
\begin{eqnarray*}
\jac{5+12k}{1+2k} & = & \jac{-1}{1+2k} \\
& = & (-1)^k.
\end{eqnarray*}
By quadratic reciprocity,
\begin{eqnarray*}
\jac{1+2k}{5+12k}\jac{5+12k}{1+2k}&=&(-1)^{((1+2k)-1)((5+12k)-1)/4}
\\
 & = & 1.
\end{eqnarray*}
Putting these together we obtain
$$
\jac{2+4k}{5+12k} = (-1)^{k+1}(-1)^k=-1.
$$
implying that $2+4k$ is not a square modulo $5+12k$.  Consequently,
 $\ord{5+12k}{3}$ is even as desired.  Hence $r$ is
$(3,1)$-distinguished.
\end{proof}

We note that trying to employ a similar argument for $40$ modulo
$96$, one runs into the problem of trying to show that $(2+4k)^2$ is
not a square modulo $5+12k$, which is clearly ridiculous.

\section{Distinguished with Respect to $(s,t)$} \label{st} \setcounter{equation}{0}

In this section we analyze the general case.  As we shall see, the
 $(s,t)$ case is more complicated than the $(3,1)$ case, in
part because $s-1$ can be composite.  This leads to potential
difficulties in calculating $\ord{r}{s}$.  On the bright side,
however, allowing $t\ne 1$ can sometimes make it easier for a number
$r$ to be distinguished.  We begin this section by examining the
role of $t$.

\begin{lemma}\label{lem:distdescent}
Let $r \ge 2$ be relatively prime to $s$. If $r$ is distinguished
with respect to $(s,t)$, then $r$ is distinguished with respect to
$(s,\gcd(t,s-1))$.
\end{lemma}

\begin{proof}
Note that $s^{\ord{r}{s}}-1 = (s-1)(1 + s + \cdots +
s^{\ord{r}{s}-1})$, and so $r \mid (s-1)(1 + s + \cdots +
s^{\ord{r}{s}-1})$. Since $r$ is distinguished with respect to
$(s,t)$, it follows that $r \mid t (1 + s + \cdots +
s^{\ord{r}{s}-1})$.  Thus, $r \mid \gcd(t,s-1) (1 + s + \cdots +
s^{\ord{r}{s}-1})$. \end{proof}

We now generalize Lemma~\ref{ord31reduction}.

\begin{prop} \label{prop:ordreduction}
Let $r \ge 2$ be relatively prime to $s$. Then $r$ is distinguished
with respect to $(s,t)$ if and only if $\ord{r}{s} = \ord{gr}{s}$
where \ba\label{def:g} g = \frac{s-1}{\gcd(s-1,t)}.\ea
\end{prop}

\begin{proof}
If  $r$ is distinguished with respect to $(s,t)$, then by Lemma
\ref{lem:distdescent}, $r$ is distinguished with respect to $(s,
\gcd(t,s-1))$, and so
   $(s-1)r \mid \gcd(s-1,t)
(s^\ord{r}{s}-1)$. From this, we see
 $gr \mid s^\ord{r}{s}-1$, and so
$s^\ord{r}{s} \equiv 1 \mod gr$.
 Thus $\ord{r}{s} \ge \ord{gr}{s}$, and since the
reverse inequality always holds,  $\ord{r}{s} = \ord{gr}{s}$.

Conversely, suppose $\ord{r}{s} = \ord{gr}{s}$.  Since
$s^\ord{gr}{s} \equiv 1 \mod gr,$ it follows that $s^\ord{r}{s}
\equiv 1 \mod gr,$ and so $gr \mid s^\ord{r}{s} -1$.  From this, it
follows that $r \mid \gcd(s-1,t) (s^\ord{r}{s} - 1)/(s-1)$, and so
$r \mid t (s^\ord{r}{s} -1)/(s-1)$.
\end{proof}

We now have the following corollary:

\begin{cor}
If $\gcd(s-1,t) \mid \gcd(s-1,t')$ (in particular, if $ t \mid t'$),
then $r$ is distinguished with respect to $(s,t')$ whenever $r$ is
distinguished with respect to $(s,t)$.
\end{cor}

In the $(3,1)$ case, we were fortunate that $s-1$ was prime and
$g=1$, which simplified our work.  We now turn to generalizing the
second part of Lemma~\ref{dist3}, and afterwards, we shall then
generalize its first part.

\begin{prop} \label{prop:orddiv}
Let $p_1,\dots,p_n$ be the prime divisors of $g=p_1^{j_1}\dots
p_n^{j_n}$. For $r\in\mathbb{Z}$ with $r=p_1^{m_1}\dots p_n^{m_n} k$
with $\gcd(k,g)=1$.  If $\ord{p_i^{m_i+j_i}}{s}$ divides
$\ord{r}{s}$ for $i=1,\dots,n$, then $r$ is $(s,t)$-distinguished.
In particular, if $\ord{p_i^{m_i+j_i}}{s}$ divides $\ord{k}{s}$,
then $r$ (and $k$) are $(s,t)$-distinguished.
\end{prop}
\begin{proof}
Given $r$ as above, we calculate $\ord{gr}{s}$.  Using prime
factorizations, we have
\begin{eqnarray*}
\ord{gk}{s} & = &
\lcm(\ord{p_1^{m_1+j_1}}{s},\dots,\ord{p_n^{m_n+j_n}}{s},\ord{k}{s}).
\end{eqnarray*}
Since $k|r$ implies $\ord{k}{s}|\ord{r}{s}$, it follows that
$\ord{gr}{s}\le\ord{r}{s}$.  However, since the latter divides the
former, we must have $\ord{gr}{s}=\ord{r}{s}$.  Consequently by
Proposition~\ref{prop:ordreduction}, $r$ is $(s,t)$-distinguished.
\end{proof}

There are now two basic possibilities for how $r$ can be
$(s,t)$-distinguished. First, if
$\ord{p_i^{m_i+j_i}}{s}=\ord{p_i^{m_i}}{s}$ then $p_i$ imposes no
restriction on $k$. This is what happened in the $4k$ case for
$(3,1)$ as $\ord{4}{3}=\ord{8}{3}$ (it also occurs in the $1$ and
$5$ modulo $6$ cases). Alternatively, if
$\ord{p_i^{m_i+j_i}}{s}>\ord{p_i^{m_i}}{s}$ then it is necessary
that $\ord{r/p_i^{m_i}}{s}$ is a multiple of
$\ord{p_i^{m_i+j_i}}{s}$.  In the $(3,1)$ case, this reduced to
$\ord{k}{3}$ (in the $r=2k$ case), as $s-1=2$.

Note that if $k$ is relatively prime to $gs$, then $k$ is
necessarily $(s,t)$-distinguished since
$\ord{gk}{s}=\lcm(\ord{g}{s},\ord{k}{s})$, and $g|(s-1)$ implies
$\ord{g}{s}=1$.  In the remainder of this section, we analyze what
multiples of $k$ are $(s,t)$-distinguished in this case.  The
following lemma, which follows from Proposition~\ref{prop:orddiv},
allows us to reduce to considering $\ord{p^{\ell}}{s}$ and its
relationship to $\ord{k}{s}$.

\begin{lemma} \label{primedivide}
Let $g=p_1^{j_1}\dots p_n^{j_n}$ be defined as in
Proposition~\ref{prop:ordreduction} with each $p_i$ prime,
$\bar{g}=p_1\dots p_n$  and let $r=p_1^{t_1}\dots p_n^{t_n} k$ with
$\gcd(gs,k)=1$. Then $r$ is distinguished if and only if for each
$i=1,\dots,n$ either
\begin{enumerate}
\item{} $\ord{p_i^{j_i+t_i}}{s} = \ord{p_i^{t_i}}{s}$ or
\item{} $\ord{p_i^{j_i+t_i}}{s}$ divides $\ord{k}{s}$.
\end{enumerate}
\end{lemma}

Similar to what was done in the $(3,1)$ case,
Proposition~\ref{prop:orddiv} allows us to create congruence classes
of $(s,t)$-distinguished positive integers.

\begin{prop}\label{stprogression}
Suppose $r$ is $(s,t)$-distinguished with $g$ as defined in equation
(\ref{def:g}). Let $g=p_1^{j_1}\dots p_n^{j_n}$ be the prime
factorization of $g$, and set $\bar{g}=p_1\dots p_n$.  Suppose
$$
r'\equiv br \mod\bar{g}rs
$$
with $\gcd(b,\bar{g}s)=1$, then $r'$ is $(s,t)$-distinguished.
\end{prop}
\begin{proof}
Since $r$ is $(s,t)$-distinguished, by
Proposition~\ref{prop:ordreduction} we have
$$
\ord{gr}{s}=\ord{r}{s}.
$$
Suppose $r=\gamma k$ where $\gcd(k,g)=1$ and any prime dividing
$\gamma$ divides $g$. Since $r'=br+\bar{g}r\!sc=\gamma
k(b+\bar{g}sc)$ for some integer $c$ and
$\gcd(\gamma,k(b+\bar{g}sc))=1$, it follows that
$$
\ord{r'}{s}=\lcm(\ord{\gamma}{s},\ord{k(b+\bar{g}sc)}{s}).
$$
Similarly,
$$
\ord{gr'}{s}=\lcm(\ord{g\gamma}{s},\ord{k(b+\bar{g}sc)}{s}).
$$
By hypothesis, however,
$$
\lcm(\ord{\gamma}{s},\ord{k}{s})  =\ord{r}{s}=\ord{gr}{s}=
\lcm(\ord{g\gamma}{s},\ord{k}{s}).
$$
Since $\ord{k}{s}$ divides $\ord{k(b+\bar{g}sc)}{s}$, it follows
that the above implies
$$
\lcm(\ord{\gamma}{s},\ord{k(b+\bar{g}sc)}{s})=
\lcm(\ord{g\gamma}{s},\ord{k(b+\bar{g}sc)}{s})
$$
implying $\ord{r'}{s}=\ord{gr'}{s}$.
Proposition~\ref{prop:ordreduction} then implies $r'$ is
$(s,t)$-distinguished as desired.
\end{proof}

This allows us to generalize Corollary \ref{31progression} to the
$(s,t)$ case.

\begin{cor}
The set of all $(s,t)$-distinguished integers can be written as an
infinite union of arithmetic progressions.
\end{cor}

Applying Proposition \ref{stprogression} to the case $r=10$ for
$(s,t)= (3,1)$, we obtain that every term in the arithmetic
progression $(10+60m)_{m=1}^\infty$ is $(3,1)$-distinguished.  We
note that this does not contain all $(3,1)$-distinguished arithmetic
progressions that include 10.  In particular, in Proposition
\ref{10mod24} we showed that the progression $(10+24m)_{m=1}^\infty$
is $(3,1)$-distinguished.  An interesting question is whether for a
given $\alpha$, one could determine all ``minimal" $\mu$ such that
the progression $(\alpha+\mu m)_{m=1}^\infty$ is
$(s,t)$-distinguished.

Proposition \ref{stprogression} generalize the argument that $40$
lies in an arithmetic sequence of $(3,1)$-distinguished integers.
The prime divisor $5$ played an important role in the argument for
40. Looking at the chart of $(3,1)$-distinguished integers, we see
that $5$, $10$, $20$, and $40$ are all $(3,1)$-distinguished, but
that $2^t5$ is not for any $t>3$. This seems curious.  Looking at
$7$, we note that $7$, $14$, and $28$ are distinguished, but again
$2^t7$ does not seem to be distinguished if $t>2$.  One might be
tempted to conjecture that for each prime $p>3$ there exists a $t_0$
such that $2^tp$ is $(3,1)$-distinguished for $t<t_0$ but is not
distinguished if $t\ge t_0$.  However, $11$ and $44$ are
$(3,1)$-distinguished, but $22$ is not.  In the remainder of the
section, we analyze this phenomenon.

We now generalize Proposition~\ref{ord2ell} and the first part of
Lemma~\ref{dist3}.

\begin{prop}  \label{lambdajump}
Let $p$ be a prime that does not divide $s$.  For $\ell \ge 1$,
define $\lambda(\ell,p,s) = \ord{p^{\ell+1}}{s}/\ord{p^{\ell}}{s}$.
Then we have the following.
\begin{enumerate}
\item[(a)] For $\ell \ge 1$, $\lambda(\ell,p,s) \mid p$.
\item[(b)] For $\ell \gg
0$, $\lambda(\ell,p,s) = p$.
\end{enumerate}
\end{prop}

\begin{proof}
Since by definition, $s^\ord{p^\ell}{s} \equiv 1 \mod p^\ell$, it
follows that for any non-negative integer $i$,
$s^{i\cdot\ord{p^\ell}{s}} \equiv 1 \mod p^\ell$, and so $p \mid
\sum_{i=0}^{p-1}s^{i\cdot\ord{p^\ell}{s}}$.  By definition, $p^\ell
\mid \left( s^{\ord{p^\ell}{s}} - 1\right)$, and since $s^{p \cdot
\ord{p^\ell}{s}} -1 = \left( s^{\ord{p^\ell}{s}} - 1\right)
\left(\sum_{i=0}^{p-1}s^{i\ord{p^\ell}{s}}\right)$, we have  $s^{p
\cdot\ord{p^{\ell}}{s}}{s} \equiv 1 \mod p^{\ell+1}$.  However, by
definition, $\ord{p^{\ell+1}}{s}$ is the smallest exponent such that
$s^\ord{p^{\ell+1}}{s} \equiv 1 \mod p^{\ell+1}$, and so
$\ord{p^{\ell+1}}{s} \mid p \cdot \ord{p^{\ell}}{s}$.  Thus, \ba
\lambda(\ell,p,s) \mid p. \ea

For any positive integer $m$ and prime $p$, we define the valuation
$\nu_p: \N^+ \to \N$ by $\nu_p(m) = j$ where $m$ can be factored as
$m=p^jn$ such that $n$ is not divisible by $p$.  Define $\delta_\ell
=\nu_{p}\left(s^{\ord{p^{\ell}}{s}}-1 \right) - \ell$. By
definition, $p^\ell \mid s^{\ord{p^{\ell}}{s}}-1 $, and so
$\delta_\ell \ge 0$. Note that \begin{eqnarray*} \label{plplus1}
s^{\ord{p^{\ell+1}}{s}}-1 &  = & s^{\lambda(\ell,p,s)
\ord{p^\ell}{s}}-1
 = (s^{\ord{p^\ell}{s}}-1) \left(
\sum_{i=0}^{\lambda(\ell,p,s)-1} s^{i \cdot \ord{p^\ell}{s}}
\right),\end{eqnarray*} and so \begin{eqnarray*} \delta_{\ell} -
\delta_{\ell+1}
& = & 1 + \nu_{p}(s^{\ord{p^\ell}{s}}-1) -   \nu_{p}\left(s^{\ord{p^{\ell+1}}{s}}-1 \right) \\
 & = & \label{1minusnu}
1 - \nu_p \left( \sum_{i=0}^{\lambda(\ell,p,s)-1} s^{i \cdot
\ord{p^\ell}{s}} \right) . \end{eqnarray*}  Since
$s^{\ord{p^\ell}{s}} \equiv 1 \mod p^\ell,$ \ba \label{equival}
\sum_{i=0}^{\lambda(\ell,p,s)-1} s^{i \cdot \ord{p^\ell}{s}} \equiv
\lambda(\ell,p,s) \mod p^\ell.\ea and since $\lambda(\ell,p,s) \mid
p$, the summation $\sum_{i=0}^{\lambda(\ell,p,s)-1} s^{i \cdot
\ord{p^\ell}{s}}$ is not divisible by $p^2$ whenever $\ell > 1$, in
which case
$$ \label{boundvp} \nu_p\left(\sum_{i=0}^{\lambda(\ell,p,s)-1} s^{i \cdot \ord{p^\ell}{s}}\right) \le 1.$$
Thus,
$\delta_\ell \ge \delta_{\ell +1}$ for $\ell >1$,  and so the
sequence $\{ \delta_i\}$ must stabilize; that is, $\delta_\ell =
\delta_{\ell+1}$ for $\ell \gg 0$, and so \begin{eqnarray*} \nu_p
\left( \sum_{i=0}^{\lambda(\ell,p,s)-1} s^{i \cdot \ord{p^\ell}{s}}
\right) = 1. \end{eqnarray*} Combining this with (\ref{equival}), it
follows that $\lambda(\ell,p,s) = p$.
\end{proof}

In light of this result, it behooves us to define the stabilization
point.
\begin{definition}
Let $s$ be a positive integer not divisible by the prime $p$. Define
$\alpha_{s,p}$ to be the smallest integer such that for any $\ell
\ge \alpha_{s,p}$, $\lambda(\ell,p,s) = p$.
\end{definition}

In the cases of interest to us, the prime $p$ divides $g$ and hence
$s-1$.  In this case, we can say more about $\lambda(\ell,p,s)$.  In
particular, if $p=2$ and $\nu_p(\ord{p^\ell}{s})>1$, then
we shall see $\lambda(\ell,p,s)=2$. Similarly, if $p>2$ is prime, then
$\lambda(\ell,p,s)=p$ if $\nu_p(\ord{p^\ell}{s})\ge1$. That is, once
the $p$ part of the $p^\ell$-order of $s$ is $p$ (or $4$ if $p=2$),
then the order increases by a factor of $p$ each time the power of
the modulus increases by $1$.  Consequently, for primes greater than
$2$, if we know $\alpha_{s,p}$, we can easily determine
$\nu_{p^\ell}(s)$ for all $\ell$.  We prove this in the remainder of
the section.

\begin{lemma}
If $\lambda(\ell,p,s)=p$ for some $\ell\ge2$ then $\lambda(m,p,s)=p$
for all $m\ge\ell$.
\end{lemma}
\begin{proof}
Let $t=\nu_p(\ord{p^\ell}{s})$.  Then we can write $\ord{p^\ell}{s}$
as  $xp^t$ (where $\gcd(x,p)=1$).  We now have
\begin{eqnarray}
s^{xp^t}&\equiv& 1\mod p^\ell \label{eq:cong}, \\
s^{xp^t} & \not\equiv  & 1 \mod p^{\ell+1}. \quad\mbox{and}
\label{eq:notcong}\\
s^{xp^{t+1}} & \equiv& 1 \mod p^{\ell+1}, \label{eqcongellplus1}
\end{eqnarray}
with the last two coming from our assumption that
$\lambda(\ell,p,s)=p$.

 Equations~(\ref{eq:cong}) and~(\ref{eq:notcong}) imply
$s^{xp^t}-1\equiv yp^\ell \mod p^{\ell+1}$ for some $y$ relatively
prime to $p$.  We then have
$$
s^{xp^{t+1}}-1=\left(s^{xp^t}-1\right)\left(1+s^{xp^t}+s^{2xp^t}+\cdots+s^{(p-1)xp^t}\right).
$$
Therefore
\begin{eqnarray*}
\nu_p\left(s^{xp^{t+1}}-1\right) & = & \nu_p\left(s^{xp^t}-1\right)
+ \nu_p\left(1+s^{xp^t}+s^{2xp^t}+\cdots+s^{(p-1)xp^t}\right) \\
& = & \ell +
\nu_p\left(1+s^{xp^t}+(s^{xp^t})^2+\cdots+(s^{xp^t})^{p-1}\right) \\
& = & \ell+1.
\end{eqnarray*}
The second equality holds by equations~(\ref{eq:cong})
and~(\ref{eq:notcong}).  The last equality holds
because~(\ref{eq:cong}) implies each of the $p$ terms is congruent
to $1$ modulo $p^2$ (since $\ell\ge2$).  Hence $p^{\ell+2}$ does not
divide $s^{xp^{t+1}}-1$ and Proposition~\ref{lambdajump} implies
$\lambda(\ell+1,p,s)=p$.  By induction $\lambda(m,p,s)=p$ for all
$m\ge\ell$.
\end{proof}

In the above proof, the requirement that $\ell\ge 2$ was only used
in arguing that
$\nu_p\left(1+s^{xp^{t}}+\dots+s^{(p-1)xp^{t}}\right)=1$.  If
$\ell=1$, we have only that this term is congruent to $p$ modulo
$p$, and thus might be $0$ modulo $p^2$ or $p^3$, etc.  On the other
hand, if we knew further that $s=1+py$ for some integer $y$, by the
binomial theorem
$$
s^p-1=(1+py)^p-1=\sum_{k=1}^p {p\choose k}(py)^k.
$$
Since $p\choose k$ is divisibly by $p$ for $1\le k<p$, if the prime
$p$ is greater than $2$, $\nu_p\left({p\choose k}
(py)^k\right)>2+\nu_p(y)$ when $2\le k\le p$. Since
$\nu_p\left({p\choose 1}(py)\right)=2+\nu_p(y)$, it follows that
$\nu_p(s^p-1)=2+\nu_p(y)$.  However, $\nu_p(s-1)=1+\nu_p(y)$. As a
result, for $p>2$ and $\ell=1+\nu_p(y)$, we have
$\lambda(\ell+1,p,s)=\lambda(\ell,p,s)=p$.  However, this implies
that $\nu_p(\ord{p^{\ell+1}}{s})=2$, and thus
$\alpha_{s,p}=\ell=\nu_p(s-1)$.  A straightforward argument now
shows the following result:
\begin{prop}
Let $s$ be a positive integer and $p$ be a prime divisor of $s-1$.
Then
\begin{enumerate}
\item If $p>2$, then
$$
\nu_p(\ord{p^\ell}{s})=\left\{\begin{array}{cl} 0 &
\ell \le \nu_p(s-1) \\
\ell-\nu_p(s-1) & \ell> \nu_p(s-1). \end{array}\right.
$$
\item If $p=2$, then
$$
\nu_2(\ord{2^\ell}{s})=\left\{\begin{array}{cl}
   0  & \ell\le \nu_2(s-1) \\
   1  & \nu_2(s-1) < \ell\le \nu_2(s^2-1) \\
   1+\ell-\nu_2(s^2-1) & \ell> \nu_2(s^2-1).  \end{array}\right.
$$
\end{enumerate}
\end{prop}
\smallskip

Now, suppose that $g=p_1^{j_1}\dots p_n^{j_n}$ with each $p_i$ prime
(and $j_i\ge1$) is defined as in
Proposition~\ref{prop:ordreduction}. For a given $k$ with
$\gcd(k,gs)=1$, we now determine for which values of $t_1,\dots,t_n$
we have that $p_1^{t_1}\dots p_n^{t_n} k$ are $(s,t)$-distinguished.

\begin{prop}\label{distbounds}
Suppose $s>t>0$ are integers, and let $g$ be defined as above with
$g=p_1^{j_1}\dots p_n^{j_n}$ the prime factorization of $g$. Then
$p_1^{t_1}\dots p_n^{t_n}k$ is $(s,t)$-distinguished if and only if
for each $i=1,\dots,n$, the power $t_i$ satisfies one of the
following:
\begin{enumerate}
\item $j_i+t_i\le\nu_{p_i}(s-1)$,
\item $p_i=2$ and $\nu_2(s-1) < t_i<t_i+j_i \le \nu_2(s^2-1)$, or
\item $\ord{p_i^{t_i+j_i}}{s}$ divides $\ord{k}{s}$.
\end{enumerate}
\end{prop}
\begin{proof}
Since $\ord{p}{s}=1$, the above follows from
Lemma~\ref{primedivide}.
\end{proof}

To demonstrate Proposition \ref{distbounds} in practice, we consider
the case where $(s,t) = (11,1)$. In the following three matrices,
the $(i,j)$-entry represents the whether or not $2^i5^jk$ is
distinguished, where $k=51, 101,$ and $151$, respectively.  Here
$g=p_1^{j_1} \cdot p_2^{j_2}$  where $p_1 = 2$, $p_2 = 5$, and $j_1
= j_2 = 1$.   We note that $\nu_2(s^2-1) = \nu_2(120) = 3$,
$\nu_2(s-1) = \nu_2(10) = 1$
 and
$\nu_5(s-1) = \nu_5(10) = 1$, while $\ord{51}{11} = 2^4$,
$\ord{101}{11} = 2^25^2$, and $\ord{151}{11} = 3\cdot5^2$. Computing
we obtain the following three charts, where a `y' denotes that
$2^i5^jk$ is $(11,1)$-distinguished and an `n' denotes that it is
not.

\vskip 0.3cm

{ \begin{center}
\begin{tabular}{cc}
\begin{tabular}{|c||c|c|c|c|c|c|c|} \hline
i$\backslash$j & 0 & 1 & 2 & 3 & 4 & 5 & 6 \\ \hline \hline 0   &
\multicolumn{1}{>{\columncolor[rgb]{0.7,0.7,0.7}}c|}{y} & n & n & n
& n & n & n \\ \hline 1   &
\multicolumn{1}{>{\columncolor[rgb]{0.7,0.7,0.7}}c|}{y} & n & n & n
& n & n & n
\\ \hline 2   &  \multicolumn{1}{>{\columncolor[rgb]{0.7,0.7,0.7}}c|}{y} & n & n & n & n & n & n \\
\hline 3 & \multicolumn{1}{>{\columncolor[rgb]{0.7,0.7,0.7}}c|}{y} & n & n & n & n & n & n \\ \hline 4   &  \multicolumn{1}{>{\columncolor[rgb]{0.7,0.7,0.7}}c|}{y} & n & n & n &
n & n & n \\ \hline 5 &
\multicolumn{1}{>{\columncolor[rgb]{0.7,0.7,0.7}}c|}{y} & n & n & n
& n & n & n
\\ \hline 6 &  n & n & n & n & n & n & n \\ \hline
\end{tabular}
&
\begin{tabular}{|c||c|c|c|c|c|c|c|} \hline
i$\backslash$j & 0 & 1 & 2 & 3 & 4 & 5 & 6 \\ \hline  \hline0   &
\multicolumn{1}{>{\columncolor[rgb]{0.7,0.7,0.7}}c|}{y} &
\multicolumn{1}{>{\columncolor[rgb]{0.7,0.7,0.7}}c|}{y} &
\multicolumn{1}{>{\columncolor[rgb]{0.7,0.7,0.7}}c|}{y}  & n & n & n
& n \\ \hline  1   &
\multicolumn{1}{>{\columncolor[rgb]{0.7,0.7,0.7}}c|}{y} &
\multicolumn{1}{>{\columncolor[rgb]{0.7,0.7,0.7}}c|}{y} &
\multicolumn{1}{>{\columncolor[rgb]{0.7,0.7,0.7}}c|}{y} & n & n
& n & n \\
\hline 2   &
\multicolumn{1}{>{\columncolor[rgb]{0.7,0.7,0.7}}c|}{y} &
\multicolumn{1}{>{\columncolor[rgb]{0.7,0.7,0.7}}c|}{y} &
\multicolumn{1}{>{\columncolor[rgb]{0.7,0.7,0.7}}c|}{y}  & n & n & n & n \\
\hline 3   & \multicolumn{1}{>{\columncolor[rgb]{0.7,0.7,0.7}}c|}{y}
& \multicolumn{1}{>{\columncolor[rgb]{0.7,0.7,0.7}}c|}{y} &
\multicolumn{1}{>{\columncolor[rgb]{0.7,0.7,0.7}}c|}{y} & n & n & n
& n \\ \hline 4   & n & n & n & n & n & n & n \\ \hline 5   &  n & n
& n & n &
n & n & n \\ \hline 6   &  n & n & n & n & n & n & n \\
\hline
\end{tabular} \vspace{4pt} \\
$k=51$ & $k=101$ \\
\end{tabular}

\vskip 0.3 cm

\begin{tabular}{|c||c|c|c|c|c|c|c|} \hline
i$\backslash$j & 0 & 1 & 2 & 3 & 4 & 5 & 6 \\ \hline \hline 0   &
\multicolumn{1}{>{\columncolor[rgb]{0.7,0.7,0.7}}c|}{y}&
\multicolumn{1}{>{\columncolor[rgb]{0.7,0.7,0.7}}c|}{y} &
\multicolumn{1}{>{\columncolor[rgb]{0.7,0.7,0.7}}c|}{y}& n & n & n &
n \\ \hline  1   &  n & n & n & n & n & n & n \\ \hline 2   &
\multicolumn{1}{>{\columncolor[rgb]{0.7,0.7,0.7}}c|}{y}&
\multicolumn{1}{>{\columncolor[rgb]{0.7,0.7,0.7}}c|}{y} &
\multicolumn{1}{>{\columncolor[rgb]{0.7,0.7,0.7}}c|}{y}& n & n & n & n \\
\hline 3   &  n & n & n & n & n & n & n \\ \hline 4   &  n & n & n &
n & n & n & n \\ \hline 5   &  n & n & n & n & n
& n & n \\ \hline 6   &  n & n & n & n & n & n & n \\
\hline
\end{tabular}
\vspace{4pt} \\

$k=151$

\end{center}
}

In fact, since either $\nu_2(s-1) =1$ or $\nu_2(s+1) = 1$, an easy
argument shows that such a chart will always have at most one ``gap"
from a shaded rectangle.

\section{Another Spanning Set for $\phi_{s,t}$} \label{spanset} \setcounter{equation}{0}

As seen in Theorem \ref{edbasis}, the functions of the form
$\psi_{s,t,r,n}$ together with $1/(1-x)$ span the vector space of
rational functions fixed by $\phi_{s,t}$.  In this section, we
generate another spanning set for this vector space.

Consider the map $\rho_{s,t}: \Z_r \to \Z_r$ given by $n \mapsto sn
+ t \mod r$.
Given $n, \, r \in \mathbb{N}$ with  $r \ge 1$ such that $r$ and $s$
are relatively prime, define \ba F_{s,t,r,n} & = & \{
\rho_{s,t}^{(i)}(n) \text{ mod } \, r: \; i \in \mathbb{Z} \}, \ea
where $\rho_{s,t}^{(i)}$ represents the $i$-th iterate of the
function $\rho_{s,t}$.

Define $\Upsilon_{s,t,r}$ to be a complete collection of coset
representatives (all chosen to be less than $r$); i.e., for all $n
\in \N$, there exists a unique $n' \in \Upsilon_{s,t,r}$ such that
$F_{s,t,r,n} = F_{s,t,r,n'}$.  For $r \ge 1$, define \ba
\curt_{s,t,r,n}(x)  = \frac{1}{1-x^r}
 \sum_{j\in F_{s,t,r,n}}
x^j . \ea

For example, consider the case when $s=3, t=1, r=13$.  Then we have
the following cosets: $F_{3,1,13,0} = \{0,1,4\}, F_{3,1,13,2} =
\{2,7,9\}, F_{3,1,13,5} = \{3,10,5\}, F_{3,1,13,6} = \{6\},$ and
$F_{3,1,13,8} = \{8,12,11\}$.  This produces the following rational
functions:
\begin{eqnarray*}
\curt_{3,1,13,0} &  =  &  \frac{1}{1-x^{13}} \left(1 + x + x^4  \right)\\
\curt_{3,1,13,2} &  =  &  \frac{1}{1-x^{13}} \left( x^2 + x^7 + x^9 \right)\\
\curt_{3,1,13,3} &  =  &  \frac{1}{1-x^{13}} \left( x^3 + x^{10} + x^5 \right)\\
\curt_{3,1,13,6} &  =  &  \frac{1}{1-x^{13}} \left( x^6 \right)\\
\curt_{3,1,13,8} &  =  &  \frac{1}{1-x^{13}} \left(  x^7 + x^8 + x^{11} \right)\\
\end{eqnarray*}

It is clear by the definition of the map $\rho_{s,t}$ that each
rational function of the form $\curt_{s,t,r,n}$ is fixed by
$\phi_{s,t}$

\begin{theorem}
Suppose $s\ge 2$ and $1 \le t \le s-2$.  The collection of all
$\curt_{s,t,r,n}$ where $r$ is distinguished with respect to $(s,t)$
and $n \in \Upsilon_{s,t,r}$ spans  the set of all rational
functions that are fixed under the transformation $\phi_{s,t}$.
\end{theorem}

\begin{proof}
It is shown in \cite{mosteig} (see the proof of Proposition 3.2 in
that paper) that for any rational function fixed by $\phi_{s,t}$
\begin{enumerate}
\item[(i)] the degree of the numerator is less than the degree
of the denominator,
\item[(ii)] the poles must be simple, and
\item[(iii)] the poles must be roots of unity.
\end{enumerate}

Therefore, any rational function fixed by $\phi_{s,t}$ can be
expressed in the form $p(x)/(1-x^r)$ where $\deg p(x) < r$.

If $p(x)/(1-x^r) = \sum a_n x^n$ is fixed by $\phi_{s,t}$, then for
each $n\in \N$, $a_n = a_{sn+t}$.  Therefore, if $j_1, j_2 \in
F_{s,t,r,n}$, then $a_{j_1} = a_{j_2}$.  This means that
coefficients are constant over terms indexed by any given coset
$F_{s,t,r,n}$, and so $p(x)/(1-x^r)$ must be a linear combination of
rational functions of the form $\curt_{s,t,r,n}$.
\end{proof}

It appears that there is a great deal of redundancy in this spanning
set.   We note that for any choice of $s,t,r,N,a$, ${\mathcal
F}_{s,t,r,N}$ is a linear combination of functions of the form
${\mathcal F}_{s,t,ar,n}$, $n\in \N$.  In general, there appears to
be a correspondence between functions of the form ${\mathcal
F}_{s,t,r,n}$ and ${\mathcal F}_{s,t,ar,n'}$ for appropriate choices
of $n$ and $n'$.  It would be nice to see future investigations shed
light on the nature of this correspondence.

At this juncture, we pose the question of how to write these two
collections of rational functions relate to one another.  In
particular, we write each function of the form $\psi_{s,t,r,n}$ in
terms of functions of the form ${\mathcal F}_{s,t,r,n}$.   Consider
the example we began with $(s,r) = (3,13)$.
 Then we have
the following cosets: $C_{3,13,0} = \{0\}, C_{3,13,1} = \{1,3,9\},
C_{3,13,2} = \{2,6,5\}, C_{3,13,4} = \{4,12,10\},$ and $C_{3,13,7} =
\{7,8,11\}$.  This produces the following rational functions: {\tiny
\begin{eqnarray*}
\psi_{3,1,13,0} &  =  &  \frac{1}{1-x} \\
\psi_{3,1,13,1} &  =  &  \frac{ (e^{\frac{8\pi i}{13}}+1+e^{\frac{2\pi i}{13}}) + (e^{\frac{5\pi i}{13}}+e^{-\frac{7\pi i}{13}}+e^{\frac{7\pi i}{13}}+e^{-\frac{3\pi i}{13}}+e^{\frac{\pi i}{13}}+e^{-\frac{9\pi i}{13}})x + (e^{-\frac{4\pi i}{13}}+e^{-\frac{2\pi i}{13}}+e^{-\frac{10\pi i}{13}})x^2}{(1-e^{-\frac{8\pi i}{13}}x)(1-e^{\frac{2\pi i}{13}}x)(1-e^{\frac{6\pi i}{13}}x)}  \\
\psi_{3,1,13,2} &  =  &  \frac{(e^{-\frac{10\pi i}{13}}+
e^{\frac{4\pi i}{13}}+1)  +
(\lexp{}{7} + \lexp{-}{5} +\lexp{-}{3} +\lexp{-}{1} +\lexp{}{} + \lexp{-}{11} )x  +(\lexp{-}{8} + \lexp{}{6} + \lexp{-}{4})x^2  }{(1-e^{\frac{4\pi i}{13}}x)(1-e^{\frac{10\pi i}{13}}x)(1-e^{\frac{12\pi i}{13}}x)} \\
\psi_{3,1,13,4} &  =  &  \frac{(\lexp{}{6} + \lexp{}{8} + 1 )
+(\lexp{}{7} + \lexp{}{11} + \lexp{}{} + \lexp{-}{11} + \lexp{}{3} +
\lexp{-}{9} )x
+(\lexp{-}{8} + \lexp{}{10} + \lexp{}{12} )x^2  }{(1-e^{-\frac{2\pi i}{13}}x)(1-e^{-\frac{6\pi i}{13}}x)(1-e^{\frac{8\pi i}{13}}x)} \\
\psi_{3,1,13,7} &  =  &  \frac{(1 + \lexp{}{} + \lexp{}{4})  +
(\lexp{-}{11} + \lexp{}{3} + \lexp{}{5} + \lexp{}{7} + \lexp{}{9} +
\lexp{}{-3})x  +
(\lexp{-}{2} +\lexp{}{12} +\lexp{}{8} )x^2}{(1-e^{-\frac{4\pi i}{13}}x)(1-e^{-\frac{10\pi i}{13}}x)(1-e^{-\frac{12\pi i}{13}}x)} \\
\end{eqnarray*}}

Since the functions of the form ${\mathcal F}_{s,t,r,n}$ span the
collection of all fixed points of $\phi_{s,t}$ that correspond to
the distinguished number $r$, we can write each $\psi_{s,t,r,m}$ as
a linear combination of such functions; that is,
$$\psi_{s,t,r,n_i} = \sum \lambda_{ij} {\mathcal F}_{s,t,r,m_j}$$
for an appropriate choice of constants $\lambda_{ij} \in \C$, where
 $$n_1 = 0,\  n_2= 1, \ n_3= 2,\  n_4=4,\  n_5=7$$ and $$m_1=  0, \ m_2=2,\
 m_3=3,\
m_4=6,  \ m_5=8.$$
   Below
is a change of basis matrix $M$ whose  $(i,j)$-entry is
 $\lambda_{ij}$.
{\tiny $$\left[
\begin{array}{ccccc}
 1 & 1 & 1 & 1 & 1 \\
 1+e^{\tiny{\frac{2 i \pi }{13}}}+e^{\frac{8 i \pi }{13}} & e^{\frac{4 i \pi }{13}}+e^{-\frac{8 i \pi }{13}}+e^{-\frac{12 i \pi }{13}} & e^{-\frac{6 i \pi
}{13}}+e^{\frac{6 i \pi }{13}}+e^{\frac{10 i \pi }{13}} & 3
e^{\frac{12 i \pi }{13}} & e^{-\frac{2 i \pi }{13}}+e^{-\frac{4 i
\pi }{13}}+e^{-\frac{10
i \pi }{13}} \\
 1+e^{\frac{4 i \pi }{13}}+e^{-\frac{10 i \pi }{13}} & e^{\frac{2 i \pi }{13}}+e^{\frac{8 i \pi }{13}}+e^{\frac{10 i \pi }{13}} & e^{-\frac{6 i \pi
}{13}}+e^{-\frac{12 i \pi }{13}}+e^{\frac{12 i \pi }{13}} & 3
e^{-\frac{2 i \pi }{13}} & e^{-\frac{4 i \pi }{13}}+e^{\frac{6 i \pi
}{13}}+e^{-\frac{8
i \pi }{13}} \\
 1+e^{\frac{6 i \pi }{13}}+e^{\frac{8 i \pi }{13}} & e^{\frac{4 i \pi }{13}}+e^{-\frac{6 i \pi }{13}}+e^{-\frac{10 i \pi }{13}} & e^{-\frac{2 i \pi
}{13}}+e^{\frac{2 i \pi }{13}}+e^{-\frac{12 i \pi }{13}} & 3
e^{-\frac{4 i \pi }{13}} & e^{-\frac{8 i \pi }{13}}+e^{\frac{10 i
\pi }{13}}+e^{\frac{12
i \pi }{13}} \\
 1+e^{\frac{4 i \pi }{13}}+e^{-\frac{12 i \pi }{13}} & e^{\frac{2 i \pi }{13}}+e^{-\frac{4 i \pi }{13}}+e^{-\frac{6 i \pi }{13}} & e^{-\frac{8 i
\pi }{13}}+e^{-\frac{10 i \pi }{13}}+e^{\frac{10 i \pi }{13}} & 3
e^{\frac{6 i \pi }{13}} & e^{-\frac{2 i \pi }{13}}+e^{\frac{8 i \pi
}{13}}+e^{\frac{12 i \pi }{13}}
\end{array}
\right] $$ }

Note that the constant coefficients of the $\psi_{3,1,13,n}$ match
the first column of $M$, which is to be expected (and whose
justification is left to the reader). However, somewhat surprising
is the fact that the coefficients of $x^2$ in the numerators of each
$\psi_{3,1,13,n}$ appear as the entries of the last column of $M$.
We pose the question concerning whether such a correspondence holds
in general.

Note that each entry of this matrix is a sum of three or fewer
thirteenth roots of unity.  At this point, we represent the entries
of this matrix in a different fashion.  For each thirteenth root of
unity that appears, rewrite it in the form $e^{\frac{2ai\pi}{13}}$
where $0 \le a \le 12$.
  For example, $1+e^{\frac{4 i \pi }{13}}+e^{-\frac{12 i \pi }{13}}$
  can be written as $e^{\frac{0i\pi}{13}} +e^{\frac{4 i \pi }{13}}+e^{\frac{14 i \pi
  }{13}}$.  Then we note each integer $a$ such that
  $e^{\frac{2ai\pi}{13}}$ is a summand in the given expression.
  Continuing with our same example, we write $\{0,2,7\}$.
Applying this process to the entire matrix $M$, we obtain the
following matrix.

\begin{center}
$$M' = \left[
\begin{tabular}{ccccc}
\{0\} & \{0\} & \{0\} & \{0\} & \{0\} \\
 \{0,1,4\} & \{2,9,7\} &
\{10,3,5\} & \{6\} & \{12,11,8\}
\\ \{0,2,8\} & \{1,4,5\}& \{10,7,6\} & \{12\} & \{11,3,9\}
\\ \{0,3,4\} & \{2,10,8\} & \{12,1,7\} & \{11\} & \{9,5,6\}
\\ \{0,2,7\}  & \{1,11,10\} & \{9,8,5\} & \{3\} &  \{12,4,6\} \\
\end{tabular} \right]
$$\end{center}

It is interesting to note that the entire matrix $M'$ can be easily
obtained by scaling cosets of the form $F_{3,1,13,m}$. In particular
to obtain the $(i,j)$-entry of $M'$, consider multiplying the
entries of the coset $F_{3,1,13,m_j}$ by $n_i$ and then reduce
modulo 13. For example, consider the $(5,3)$ entry of $M'$, which is
$\{9,8,5\}$ according to the table above. We note that this entry
could have been obtained by multiplying $F_{3,1,13,m_3} =
F_{3,1,13,3} = \{3,10,5\}$ by $n_5=7$ and then reducing modulo 13.
An interesting open question is whether this sort of pattern holds
in general.

Let us consider what would have happened if we chose different coset
representatives other than  $n_1 = 0, n_2= 1, n_3= 2, n_4=4, n_5=7$.
For example, suppose we choose $n_5=8$.  Multiplying each
$F_{3,1,13,m_i}$ by $n_5=8$ and reducing modulo 13, we obtain the
following sets: $\{0, 6, 8 \}, \{4,7,3\}, \{11,2,1\},\{9\},
\{5,10,12\}$.  Note that if add 7 to each of the cosets and reduce
modulo 13 we obtain the last row of $M'$, which simply amounts to
multiplying the last row of $M$ by a scalar multiple.

The inverse matrix $M^{-1}$ also consists of entries that
are sums of roots of unity.  However, we have not found a similar type of pattern as $M$ has.  It would be interesting to investigate the inverse matrix more fully.

\bibliographystyle{amsalpha}

\end{document}